\documentclass[12pt,a4paper,english]{smfart}
\usepackage{amssymb,amsmath,amscd,stmaryrd,vmargin,fancybox,calc}
\usepackage[english]{babel}
\usepackage[all]{xy}
\theoremstyle{plain}
\newtheorem{theo}{Theorem}[section]
\newtheorem{prop}[theo]{Proposition}
\newtheorem{cor}[theo]{Corollary}
\newtheorem{lem}[theo]{Lemma}
\newtheorem{defi}[theo]{Definition}

\def \demdu#1 { {\sl Proof #1.} }
\def \goth#1 {\mathfrak{#1}}
\def \cal#1 {\mathcal{#1}}
\def \Gal {\text{\rm Gal}}
\def \Hom {\text{\rm Hom}}

\title[]{On equivariant characteristic ideals of real classes}    
\author{Thong Nguyen Quang Do}
\address{Laboratoire de Math{\'e}matiques,
UMR 6623 de l'Universit{\'e} de Franche-Comt{\'e} et du CNRS\\
16 route de Gray, 25030 Besan\c{c}on cedex, France}
\email{	thong.nguyen-quang-do@univ-fcomte.fr}

\begin{document}
\bibliographystyle{smfalpha}
\frontmatter

\begin{abstract}
Let $p$ be an odd prime, $F/{\Bbb Q}$ an abelian totally real number field, $F_\infty/F$ its cyclotomic 
${\Bbb Z}_p$-extension, $G_\infty = Gal (F_\infty / {\Bbb Q}),$ ${\Bbb A} = {\Bbb Z}_p [[G_\infty]].$
 We give an explicit description of the equivariant characteristic ideal of $H^2_{Iw} (F_\infty, {\Bbb Z}_p(m))$
 over ${\Bbb A}$ for all odd $m \in {\Bbb Z}$ by applying M. Witte's formulation of an equivariant main conjecture 
(or ``limit theorem'') due to Burns and Greither. This could shed some light on Greenberg's conjecture on the vanishing 
of the $\lambda$-invariant of $F_\infty/F.$
\end{abstract}

\maketitle

\mainmatter

\setcounter{section}{-1}
\section{Introduction}

Fix an odd prime $p$ and a Galois extension $F/k$ of totally real number fields, with group $G = Gal (F/k)$.
 Let $F_\infty = \displaystyle\mathop\cup_{n \geq 0} F_n$ the cyclotomic ${\Bbb Z}_p$-extension of $F$, 
$\Gamma = Gal (F_\infty/K)$, $G_\infty = Gal (F_\infty/k)$, $\Lambda = {\Bbb Z}_p [[\Gamma]]$, 
${\Bbb A} = {\Bbb Z}_p [[G_\infty]]$. Take $S_f = S_p \cup\, Ram(F/k)$, $S = S_\infty \cup S_f$, where $S_\infty$ 
(resp. $S_p)$ is the set of archimedean primes (resp. primes above $p)$ of $k$ and $Ram\ (F/k)$ is the set of places of $k$ which ramify in $F/k$. By abuse of language, we also denote by $S$ the set of primes above $S$ in any extension of 
$k$. Let $G_S (F_n)$ be the Galois group over $F_n$ of the maximal algebraic $S$-ramified (i.e. unramified outside $S$)
extension of $F_n$. Since $p \not= 2$, $cd_p (G_S (F_n)) \leq 2$ and since $S$ contains $S_p$, the continuous 
cohomology groups $H^i_S (F_n, {\Bbb Z}_p (m)) := H^i (G_S (F_n), {\Bbb Z}_p (m))$, $m \in {\Bbb Z}$, 
coincide with the \'etale cohomology groups $H^i_{\text{\'et}} ({\mathcal {O}}_{F_n} [1/S], {\Bbb Z}_p (m))$. 
We are interested in the ${\Bbb A}$-modules 
$H^i_{Iw, S} (F_\infty, {\Bbb Z}_p (m)) := \displaystyle\lim_{\underset{\text{cores}} {\longleftarrow} } 
H^i_S (F_n, {\Bbb Z}_p (m))$ ($H^i_{Iw}(\ldotp)$ for short if there is no ambiguity on $S$), to which can be attached invariants containing important arithmetical information.

For instance, if $G$ is {\it abelian} and the $\mu$-invariant associated to $F_\infty/F$ vanishes, it can be shown that, for any $m \equiv 0$ (mod 2), the initial ${\Bbb A}$ - Fitting ideal of $H^2_{Iw} (F_\infty, {\Bbb Z}_p (m))$ is given by the formula~:
\[ Fit_{\Bbb A} \bigl( H^2_{Iw} (F_\infty , {\Bbb Z}_p (m) \Bigl) = tw_m (I G_\infty \ldotp \theta_S) \eqno(1)\]
where $tw_m$ denotes the automorphism of the total ring of fractions $Q {\Bbb A}$ induced by the $m^{th}$-Iwasawa twist $\sigma \mapsto \kappa^m_{cyc}\  \sigma^{-1},$ $\kappa_{cyc}$ is the cyclotomic character, $I G_\infty$ is the augmentation ideal of ${\Bbb Z}_p [[G_\infty]]$ and $\theta_S \in Q\ {\Bbb A}$ is the Deligne-Ribet pseudo-measure associated with $G_\infty$ (\cite{Ng}, thms 3.1.2, 3.3.3). Note that for $m < 0$, we implicitly assume the validity of the $m^{th}$-twist of Leopoldt's conjecture; see proposition 3.2 below.

The identity (1) was shown by using the (abelian) Equivariant Main Conjecture (EMC for short) of Iwasawa theory in the formulation of Ritter and Weiss (\cite{RW1}, theorem 11). Here ``equivariant'' means that the Galois action of $G$ is taken into account. One salient feature of formula (1) is that it allows to prove by descent the $p$-part of the Coates-Sinnott conjecture, as well as a weak form of the $p$-part of Brumer's conjecture (\cite{Ng}, thms 4.3, 5.2). On the same subject, let us report the approach of Burns and Greither (\cite{BG2}, theorems 3.1, 5.1 and corollary 2) using the Iwasawa theory of complexes initiated by Kato and subsequently extended by Nekovar in \cite{Ne}. See also \cite{GP}.

After the recent proof of the non commutative EMC (\cite{Ka}, \cite{RW2}), no doubt that the above results 
could be extended to the case where $G$ is non abelian. But in this article, we are mainly interested in the 
odd twists, $m \equiv 1$ (mod 2), which are not a priori covered by the EMC. 
Since there are so many ``main conjectures'' floating around, a short explanation is in order. Let us come 
back to the classical Main Conjecture, or Wiles' theorem (WMC for short) and take $G = (1)$ for simplicity. 
The Galois group $\mathfrak {X}_\infty$ over $F_\infty$ of the maximal abelian $(p)$-ramified pro-$p$-extension of
 $F_\infty$ is a $\Lambda$-torsion module because $F$ is totally real, and the $\Lambda$-characteristic series of 
$\mathfrak {X}_\infty$ is precisely related by the WMC to the Deligne-Ribet pseudo-measure attached to the maximal
 pro-$p$-quotient of $G_S^{ab} (F)$ (where $S = S_\infty \cup S_p $). If we allow - just for this discussion - the base field $F$ to be a CM-field, then $tor_\Lambda \mathfrak{X}_\infty = {\goth X }^{+}_\infty$ 
by the weak Leopoldt conjecture, and it is related by Spiegelung (reflection theorems) to the 
``minus part'' $X_\infty^{-}$ of the Galois group over $F_\infty$ of the maximal abelian unramified 
pro-$p$-extension of $F_\infty$, so that we know the $\Lambda$-characteristic series of $X_\infty^{-}$ in terms of 
$L_p$-functions. However our intended study of odd twists actually concerns the ``plus part'' $X^{+}_\infty$, at least because if $\zeta_p \in F$, then $X'_\infty (m-1)$ is contained in $H^2_{Iw} (F_\infty, {\Bbb Z}_p (m))$ (see propos. \ref{3.2} below), where $X'_\infty$ is obtained from $X_\infty$ by adding the condition that all $(p)$-primes (hence all finite primes) must be totally split. A related problem is Greenberg's celebrated conjecture - a ``reasonable'' generalization of Vandiver's - which asserts the finiteness of $X^{+}_\infty$ (or equivalently of $X_\infty^{'+}$). Not much is known on the plus parts. Let us recall some results in the case where $k = {\Bbb Q},$ $F$ is totally real and 
$G$ is abelian~:
\begin{itemize}
\item it is well known that the WMC is equivalent to the so-called ``Gras type'' equality (because it implies the Gras conjecture)~:
\[char_\Lambda\ X_\infty = \ char_\Lambda (\overline U_\infty/\overline C_\infty) \eqno(2)\]
where $\overline U_\infty$ (resp. $\overline C_\infty$) denotes the inverse limit (w.r.t. norms) of the $p$-completions $\overline U_n$ (resp. $\overline C_n$) of the groups of units (resp. circular units) along the cyclotomic tower of $F$. But the right hand side $\Lambda$-characteristic series is not known. Besides, this Gras type equality comes from the multiplicativity of $ char_\Lambda (\ldotp)$ in the exact sequence of class-field theory relative to inertia~:
\[0 \to \overline U_\infty/\overline C_\infty \to {\mathcal{U}}_\infty/\overline C_\infty \to {\goth X }_\infty \to X_\infty \to 0,\] where ${\mathcal{U} }_\infty$ is the semi-local analogue of $\overline U_\infty$. But when passing from $\Lambda$ to ${\Bbb A},$ no straightforward equivariant generalization (of ${\Bbb A}$-characteristic series of modules and their multiplicativity) is known.
\item  Galois annihilators of $X'_\infty$ have been explicitly computed in \cite{NN} and \cite{So}, as well as Fitting ideals in the semi-simple case in \cite{NN}.
\end{itemize}
In this context, the main result of this paper will be an {\it explicit equivariant generalization of the Gras type equality} (2) {\it for all odd twists} (see thm  \ref{3.7} below).

The proof will proceed in essentially two stages~:
\begin{itemize}
 \item use a ``limit theorem'' in the style of Burns and Greither (\cite{BG1}, thm 6.1 ; this is also an EMC, but we don't call it so for fear of overload), but expressed in the framework of the Iwasawa theory of perfect torsion complexes as in \cite{W}, to relate the ${\Bbb A}$-determinant of $H^2_{Iw} (F_\infty, {\Bbb Z}_p (m))$, $m$ odd, to that of a suitable quotient of $H^1_{Iw} (F_\infty, {\Bbb Z}_p (m))$
\item use an axiomatization of a method originally introduced by Kraft and Schoof for real quadratic fields (\cite{KS}) to compute explicitly the latter determinant (which will be actually an ${\Bbb A}$-initial Fitting ideal).
\end{itemize}
\section{Generators and relations, and Fitting ideals}
We study in this section an axiomatic method for describing certain (initial) Fitting ideals by generators and relations. It was first introduced by \cite{KS} for quadratic real fields, and subsequently applied by \cite{So} to the cyclotomic field ${\Bbb Q}(\zeta_p)^+$.
\subsection{General case}
Since the process is purely algebraic, we can relax here the assumptions on $k$ and $F$ imposed in the introduction. So $F/k$ {\it will just be a Galois extension with group $G$}, and the usual Galois picture will be~:

\[\xymatrix@=22pt{ & F_\infty \\
k_\infty\ar@{-}[ur]^-H & F_n \ar@{-}[u]_-{\Gamma^{p^n}} \\
k_\infty \cap F\ar@{-}[u]\ar@{-}[r]& F\ar@{-}[u]^-{\Gamma_n} \\
k\ar@{-}[u]\ar@{-}[ru]_-{G_0=G}\ar@/_4pc/ @{-}[uru]_-{G_n} & \\
}
\]

Note that ${\Bbb A} = \Lambda [H]$ is equal to $\Lambda [G]$ if and only if $k_\infty \cap F = k$. Put 
$R_n = {\Bbb Z}_p [G_n]$, $R_{n,a} = {\Bbb Z}/p^a\ [G_n]$.

{\it Hypothesis~:} We are given a projective (w.r.t. norms) system of $R_n$-modules $V_n$ which are ${\Bbb Z}_p$-free of finite type, as well as a projective subsystem $W_n \subset V_n$ such that each $W_n$ is $R_n$-free of rank 1. In other words, there is a norm coherent system $\eta = (\eta_n),\ \eta_n \in V_n,$ such that $W_n = R_n \ldotp \eta_n$ for all 
$n \geq 0$. Without any originality, a pair $(V_n, W_n)_{n \geq 0}$ as above will be called {\it admissible}.

At the time being, we don't worry about the existence of such systems $(W_n, V_n)_n$. 
Arithmetical examples will be given later in sections 2 and 3.

{\it Goal~:} denoting by $B_n$ the quotient $V_n/W_n$, describe the module $(t\, B_n)^*$ 
(where $t(\ldotp)$ means ${\Bbb Z}_p$-torsion) by generators and relations.

In the sequel, for two left modules $X, Y$, the Galois action on $\Hom(X, Y)$ will always be defined by ${}^\sigma f(x) = \sigma (f(\sigma^{-1} x))$. The Pontryagin dual of $X$ will be denoted by $X^\ast.$

\begin{prop}\label{1.1}
For any $n \geq 0$, \[(t\, B_n)^\ast \simeq R_n/{\cal D }_n,\ \text{where}\ {\cal D }_n = \{ \displaystyle\sum_{\sigma \in G_n} f(\sigma^{-1} \eta_n) \ldotp \sigma \ ;\ f \in \Hom\, (V_n, {\Bbb Z}_p) \}\ .\]
\end{prop}
\begin{proof}
We follow essentially the argument of \cite{KS}, thm 2.4. For any $a \geq 0$, the snake lemma applied to the $p^a$-th power map yields an exact sequence~:
$$0 \to B_n [p^a] \to W_n/p^a \to V_n/p^a \to B_n/p^a \to 0$$
(the injectivity on the left is due to the fact that $V_n$ is ${\Bbb Z}_p$-free).

Applying the functor $\Hom_{G_n} (\ldotp, R_{n, a}),$ we get another exact sequence of $R_{n,a}$-modules~:
$$ \to \Hom_{G_n}\bigl( V_n, R_{n, a} \bigl) \to \Hom_{G_n}\bigl( W_n, R_{n, a} \bigl) \to \ \Hom_{G_n}\bigl( B_n [p^a], R_{n, a} \bigl) \to$$
But $R_{n, a}$ is a Gorenstein ring, which means that $R^\ast_{n, a}$ is a free $R_{n, a}$-module of rank 1, hence, for any finite $R_{n, a}$-module $M$, the canonical isomorphism \linebreak $\Hom_{R_n}\bigl(M, R^\ast_{n, a}\bigl) {\displaystyle\buildrel\sim\over\to} M^\ast$ gives rise to an $R_{n, a}$-isomorphism 
$\Hom_{R_n} \bigl( M, R_{n, a} \bigl) {\displaystyle\buildrel\sim\over\to} M^\ast$, $f \mapsto \psi \circ f$, 
where $\psi$ is a chosen $R_{n, a}$-generator $R_{n, a} \to {\Bbb Q}_p/{\Bbb Z}_p$. It follows that the functor 
$\Hom_{R_n} \bigl( \ldotp, R_{n, a} \bigl)$ is exact and we have
\[ \to \Hom_{R_n} \bigl( V_n, R_{n, a} \bigl) \to \ \Hom_{R_n} \bigl( W_a, R_{n, a} \bigl) \to B^\ast_n / p^a \to 0\]
Taking $\displaystyle \lim_{\longleftarrow\atop a}$, we derive an exact sequence~:
\[\to \Hom_{R_n} ( V_n, R_n ) {\buildrel res \over \longrightarrow} \Hom_{R_n} ( W_n, R_n) \to (t B_n)^\ast \ 0\ .\]
But $W_n = R_n \ldotp \eta_n,$ hence an element $f \in \Hom_{R_n} (W_n, R_n)$ is determined by the value $f(\eta_n)$, i.e. $\Hom_{R_n} (W_n, R_n) {\displaystyle\buildrel \sim \over \rightarrow} R_n$, $f \mapsto f(\eta_n)$, and the above exact sequence becomes~: $\Hom_{R_n} (V_n, R_n) \to R_n \to (t B_n)^\ast \to 0$.

In other words, $(t B_n)^\ast \simeq R_n / \{ f(\eta_n)\, ;\, f \in \Hom_{G_n} (V_n, R_n) \}$. The classical
isomorphism $\Hom_{{\Bbb Z}_p} (M, {\Bbb Z}_p) {\displaystyle\buildrel \sim \over \to} \Hom_{R_n} (M, R_n)$, 
$f \mapsto F$ such that $F(m) = \displaystyle\sum_{\sigma \in G_n} f(\sigma^{-1} m) \ldotp \sigma$, $\forall\, m \in M$, gives the desired result. \end{proof}
\begin{cor}\label{1.2}
$Ann_{R_n} \bigl( (t B_n)^\ast \bigl) = Fit_{R_n} \bigl( (t B_n)^\ast \bigl) = {\cal D }_n$
\end{cor}
\begin{proof} The proposition \ref{1.1} shows at the same time that $(t B_n)^\ast$ is \linebreak $R_n$-monogeneous, hence its $R_n$-annihilator and $R_n$-Fitting ideal coincide, and that $Fit_{R_n} \bigl( (t B_n)^\ast\bigl) = {\cal D }_n$. \end{proof}
\begin{defi}\label{1.3} In the sequel, we shall denote by ${\cal D } (F_\infty)$ (or ${\cal D }_\infty$ for short) the projective limit ${\displaystyle\lim_\leftarrow}\ {\cal D }_n$ w.r.t. norms (the detailed transition maps can be found e.g. in \cite{So}, proposition 1). Of course ${\cal D }_\infty$ depends on the admissible pair $(V_n, W_n)_{n \geq 0}$. The corollary \ref{1.2} shows that $Ann_{\Bbb A} ({\displaystyle\lim_\leftarrow}\ (t B_n)^\ast) = Fit_{\Bbb A} ({\displaystyle\lim_\leftarrow}\ (t B_n)^\ast) = {\cal D }_\infty$.
\end{defi}

\subsection{A kummerian description}

Let $E = F(\zeta_p)$ and $\Delta = Gal\ (E/F)$. Fix a norm coherent system $\boldsymbol{\zeta} = \bigl( \zeta_{p^n}\bigl)_{n \geq 0}$ of generators of the groups $\mu_{p^n}$ of $p^{nth}$ roots of unity. Attach to $E_\infty/E$ the following objects~: $U_n$ (resp. $U'_n)$ = group of units (resp. ($p$-units) of $E_n$
\[\overline U_n = \overline U_n \otimes {\Bbb Z}_p \ ,\ \overline U'_n = U'_n \otimes {\Bbb Z}_p\]
${\goth X }_n$ = Galois group over $E_n$ of the maximal abelian $(p)$-ramified pro-$p$-extension of $E_n$. At the level of $E_n$, we have the Kummer pairing~:
\[H^1 \bigl( G_S(E_n), \mu_{p^n} \bigl) \times {\goth X }_n/p^n \rightarrow \mu_{p^n},\,  
(x, \rho) \mapsto \Bigl( x^{{1 \over p^n}} \Bigl)^{\rho-1}\ .\]
Let $\langle \ldotp, \ldotp \rangle_n\ :\ \Hom \bigl( {\goth X }_n, {\Bbb Z}/p^n \bigl) \times {\goth X }_n / p^n
\rightarrow  {\Bbb Z}/p^n$ be the pairing defined by $\bigl( x^{{1 \over p^n}} \bigl)^{\rho -1} = 
\xi_{p^n}^{\langle x, \rho \rangle_n}$. Take $V^E_n$ to be $\overline U_n/tors$, or $\overline U'_n/tors$, or more generally a free ${\Bbb Z}_p$-module such that $V^E_n/p^n \hookrightarrow\ \Hom\ \bigl( G_S(E_n), \mu_{p^n} \bigl)$. In this subsection, we relax the condition of Galois freeness on $W^E_n$, which we assume only to be monogeneous~: 
$W^E_n = {\Bbb Z} [J_n] \ldotp \eta_n$, where $J_n = Gal (E_n/k)$.

Define $\overline{\cal D }_n (E) = \lbrace \displaystyle\sum_{\sigma \in J_n} f (\sigma^{-1} \eta_n) \ldotp \sigma$ ; 
$f \in (V_n/p^n)^\ast \rbrace = \lbrace \displaystyle\sum_{\sigma \in J_n} \langle \sigma^{-1} \ldotp \eta_n, \rho\rangle_n \sigma\ ;\ \rho \in {\goth X }_{n}/p^n \rbrace = \lbrace \displaystyle\sum_{\sigma \in G_n} \langle \eta_n, \sigma \ldotp \rho\rangle_n \, \kappa^{-1}_{cyc} (\sigma) \sigma\, ;\, \rho \in {\goth X }_n/p^n \rbrace$ (the second equality comes from the Kummer pairing just recalled; the third from a functorial property of this pairing). Recall that ${\cal D }_n (E) := \lbrace \displaystyle\sum_{\sigma \in J_n} f(\sigma^{-1} \eta_n) \ldotp \sigma\, ;\, 
f \in \Hom (V^E_n, {\Bbb Z}_p) \rbrace$.
\begin{prop}\label{1.4}
With the pair $(V^E_n, W^E_n)_{n \geq 0}$ chosen as above
\[{\cal D } (E_\infty) := {\displaystyle\lim_\leftarrow}\ {\cal D }_n (E) = {\displaystyle\lim_\leftarrow}\ \overline{\cal D }_n (E)\] is monogeneous as ${\Bbb Z}_p [[ \Gal(E_\infty / {\Bbb Q})]]$-module.
\end{prop}
\begin{proof} Since $\Hom\ (V^E_n, {\Bbb Z}_p) \simeq \bigl(V^E_n \otimes {\Bbb Q}_p/{\Bbb Z}_p \bigl)^\ast$, the equality ${\displaystyle\lim_\leftarrow}\ {\cal D }_n (E) = {\displaystyle\lim_\leftarrow}\ \overline{\cal D }_n(E)$ is straightforward. Notice next that in the sum $\displaystyle\sum_{\sigma \in J_n}  \langle \eta_n, \sigma \ldotp \rho \rangle_n\ \kappa_{cyc}^{-1} (\sigma) \sigma$, the value of the pairing $\langle \eta_n, \sigma \ldotp \rho \rangle_n$ does not depend on $\rho$, only on the image of $\rho$ in the cyclic group 
$\Gal \Bigl( E_n \bigl( \eta_n^{{1 \over p^n}}\bigl)/E_n\Bigl)$. Choosing a generator $\tau_n$ of this group, we see that the sum above is a multiple of $\displaystyle\sum_{\sigma \in J_n} \langle \eta_n, \sigma \ldotp \tau_n\rangle_n\, \kappa^{-1}_{cyc} (\sigma) \sigma$. By considering the extension of $E_\infty$ obtained by adjoining all 
$p^n$-th roots of all the $\eta_n's$ and their conjugates under the action of $Gal\ (E_\infty/{\Bbb Q})$,
we get the desired monogeneity result. 
\end{proof}
This proposition will be applied at the end of section 2 by cutting out  ${\cal D } (E_\infty)$ by characters of $\Gal(E/F)$.

\vspace{.3cm}

As usually happens, the ``new'' object ${\cal D }_\infty$ appears afterwards to be not so new. Some previous occurrences must be recorded~:
\begin{itemize}
 \item  for $ k = {\Bbb Q}$ and $E = {\Bbb Q}(\zeta_p)$ and for a special choice of 
\({\displaystyle\boldsymbol{\eta}}\), ${\cal D }_\infty$ plays an important role in Ihara's theory of universal power series for Jacobi sums (see \cite{IKY}). In \cite{NN}, $\S$ 4.1, ${\cal D }_\infty$ is interpreted as a certain module of ``$p$-adic Gauss sums'', 
and both \cite{NN} and \cite{So} show that ${\displaystyle(\mathcal D_\infty)^{\#}}$ is the ${\Bbb A}$-initial Fitting ideal of $X_\infty^{'+},$ where the sign $(\ldotp)^\#$ means inverting Galois action.
\item in exactly the same setting, ${\cal D }_\infty$ appears on the Galois side of Sharifi's conjecture on the two-variable $L_p$-function of Mazur-Kitagawa (\cite{Sh}, propos. 6.2). Note that a weak form of this conjecture has been recently proved for all fields ${\Bbb Q} (\zeta_{p^n})$ by Fukaya and Kato (\cite{FK2}).
\item for $k$ totally real and $E/k$ abelian, containing $\zeta_p$, and for a special choice of $\eta$, $(\mathcal {D}_\infty )^{\#}$ is shown in \cite{NN} and \cite{So} to annihilate $X_\infty^{' +}$. Moreover, in the semi-simple case $(p \mid\!\!\!\!\!/ \vert G \vert)$, thms 3.4.2 and 5.3.2 of \cite{NN} assert that for any odd character $\psi$ of $G$, the $\psi$-part $((\mathcal {D}_\infty )^{\#})^\psi$ is isomorphic to the Fitting ideal of $(X'_\infty)^{\psi^\ast}$, where $\psi^\ast = \omega \psi^{-1}$ denotes the ``mirror'' of $\psi,$ $\omega$ being the Teichm\"uller character.
\item for $k/{\Bbb Q}$ abelian and $F = k (\zeta_p),$ ${\cal D }_\infty$ appears prominently in the explicit reciprocity law of Coleman, generalized by Perrin-Riou (\cite{Fl} $\S$ 1.3; \cite{PR}, thm 4.3.2). Let us just recall the starting point in \cite{PR}~: denoting \linebreak $\displaystyle\mathop\oplus_{v \mid p} H^1 (E_{n,v}, {\Bbb Z}_p (m))$ by $H^1_{s\ell} (E_n, {\Bbb Z}_p(m))$ and ${\displaystyle\lim_\leftarrow}\ H^1_{s\ell} (E_n, {\Bbb Z}_p (m))$ by \linebreak $H^1_{Iw, s\ell} (E_\infty, {\Bbb Z}_p (m)),$ Perrin-Riou constructs a ``cup-product'' at the infinite level $H^1_{Iw, s\ell} (E_\infty, {\Bbb Z}_p (1)) \times H^1_{Iw, s\ell} (E_\infty, {\Bbb Z}_p) {\displaystyle\mathop\rightarrow^\cup} ({\cal O }_E \otimes {\Bbb Z}_p) [[G_\infty]]$ such that, writing $\pi_n$ for the natural projection $({\cal O }_F \otimes  {\Bbb Z}_p)\, [[G_\infty]]) \to {\Bbb Z}_p [G_n],
$, $\pi_n ({\bf x} \cup {\bf y}) = \displaystyle\sum_{\sigma \in G_n} (\sigma^{-1} x_n \cup y_n) \ldotp \sigma$ (recall that 
the cup-product 
at 
finite levels does not commute with corestriction).
\end{itemize}
\section{Semi-simple example}

In this section, to illustrate the ``Gras type'' approach via the WMC, we intend to study a (particular) semi-simple case, which the reader can skip if pressed for time. The following hypotheses will be assumed~:
Let $F/{\Bbb Q}$ be a {\it totally real abelian} extension of conductor $f$, 
such that $p$ does not divide the order of $G = Gal (F/{\Bbb Q})$.
Let $U_F$ (resp. $U'_F)$ be the group of units (resp. $(p)$-units) of $F$.
The group Cyc $(F)$ of $F$ is the subgroup of $F^\ast$ generated by $-1$ and all the elements 
$N_{{\Bbb Q}(\zeta_f)/F \cap {\Bbb Q}(\zeta_f)} (1-\zeta^a_f),$ $(a, f) = 1$.
The group $C_F$ (resp. $C'_F)$ of circular units (resp. circular $(p)$-units in Sinnott's sense is defined as $C_F = U_F \cap Cyc (F)$ 
(resp. $C'_F = U'_F \cap Cyc (F))$. Write $U_n$, $C_n$, etc for $U_{F_n}$, $C_{F_n}$, etc and 
$\overline {(\ldotp)}$ for the $p$-completion. We take $V_n = \overline U'_n$ (which is ${\Bbb Z}_p$-free of rank 
$[F_n:{\Bbb Q}]$) and look for candidates for the $W'_n$s. Since $p \mid\!\!\!\!/ \vert G \vert$, any $(p)$-place of $F$ is totally ramified in $F_\infty$. Let us denote by $s$ the number of $(p)$-places of $F$ (hence also of $F_\infty$). We keep the Galois notations of the beginning of section 1.

\begin{lem}\label{2.1}
For any $n \geq 1,$ the ${\Bbb Z}_p [G_n]$-module $\overline C'_n$ is free (necessarily of rank 1) if and only if $s = 1.$
\end{lem}
\begin{proof}
Let us first consider only the $\Lambda$-module structure. The cohomology groups $H^i(\Gamma_n, \overline C'_n)$ are
 computed in general in \cite{NLB}. In our situation, 
$\overline C'_n$ is $\Gamma_n$-cohomologically trivial if and only if 
$s = 1$ (\cite{NLB}, propos. 2.8).
Since $\Gamma_n$ is a $p$-group, ${\Bbb Z}_p  [\Gamma_n]$ is a local algebra, and for a
${\Bbb Z}_p [\Gamma_n]$-module without torsion, cohomological triviality is equivalent to 
${\Bbb Z}_p [\Gamma_n]$-freeness. To pass from $\Gamma_n$ to $G_n$, just notice that 
$H^i (G_n, \overline C'_n)^\Delta\ {\displaystyle\mathop \simeq^{\text{ res}}} H^i(\Gamma_n, \overline C'_n)$ because 
$p \mid\!\!\!\!/ \vert \Delta \vert,$ hence $s = 1$ if and only if $\overline {C}'_n$ is ${\Bbb Z}_p [G_n]$-projective 
(of ${\Bbb Z}_p$-rank equal to $\vert G_n \vert)$. By decomposing $\overline C'_n$ into $\chi$-parts, 
$\overline C'_n = {\displaystyle\mathop\oplus_{\chi \in \widehat G}} \ (\overline {C}'_n)^\chi$, we see that $s = 1$ if 
and only if each $(\overline C'_n)^\chi$ is free over the local algebra ${\Bbb Z}_p [\chi] [\Gamma_n]$, if and 
only if $\overline C'_n$ is ${\Bbb Z}_p [G_n]$-free (necessarily of rank 1). 
\end{proof}
Summarizing, if $s = 1$, one can take $V_n = \overline U'_n$, $W_n = \overline C'_n$ in order to apply proposition \ref{1.1} to $t B_n = B_n = \overline U'_n/\overline C'_n$. 
Note that in the semi-simple case with $s = 1$, $ \overline U'_n/\overline C'_n \simeq \overline U_n/\overline C_n$ for any $n \geq 1$ (\cite{NLB}, lemma 2.7) and also 
$X'_\infty \simeq X_\infty$ in the notations of the introduction (\cite{NLB}, lemma 1.5). Recall that $X_\infty$ (resp. $X'_\infty)$ is the unramified 
(resp. totally split at all finite places) Iwasawa module above $F_\infty$. We can now determine the ${\Bbb A}$-Fitting ideal of $X_\infty$ in our particular case~:
\begin{prop}\label{2.2}
 Let $F/{\Bbb Q}$ be a totally real abelian extension, such that $p \mid\!\!\!\!/ \vert G \vert$ and $s = 1$. Then $\mathcal {D} (F_\infty)^{\#} = Fit_{\Bbb A} (X_\infty),$ where
$(\ldotp)^{\#}$ means inverting the Galois action.
\end{prop}
This is a particular case $(s = 1)$ of thm 5.3.2 of \cite{NN}. Typical examples are $F = {\Bbb Q}(\zeta_p)^+,$ or $F = {\Bbb Q} (\sqrt d)$, $d$ a square free positive integer such that
\(\displaystyle \left ( \frac d p \right) \neq 1\).
\begin{proof} 
Since $B_n = \overline U_n/\overline C_n$ is finite, proposition \ref{1.1} shows that 
${\cal D } (F_\infty) = {\cal D }_\infty = Fit_{\Bbb A} \bigl( {\displaystyle\mathop{\lim_\leftarrow}}(B_n^\ast) \bigl)$. 
It remains to describe ${\displaystyle\mathop{\lim_\leftarrow}}(B_n^\ast)$. Denote $\overline U_\infty = {\displaystyle\mathop{\lim_\leftarrow}}\,  \overline U_n$, $\overline C_\infty = {\displaystyle\mathop{\lim_\leftarrow}}\,\overline C_n,$ $Y_\infty = \overline U_\infty/\overline C_\infty$, $X^0_\infty=$ the maximal finite submodule of $X_\infty$. With $s = 1$, we have the following codescent exact sequences (\cite{NLB}, proposition 4.7)~:
\[0 \longrightarrow ( Y_\infty )_{\Gamma_n} \longrightarrow B_n \longrightarrow (X^0_\infty)^{\Gamma_n} \longrightarrow 0\ .\]
Taking duals and ${\displaystyle\mathop{\lim_\leftarrow}}$ we get an exact sequence of ${\Bbb A}$-modules~:

$0 \longrightarrow (X^0_\infty)^\ast \longrightarrow {\displaystyle\mathop{\lim_\leftarrow}}\, (B^\ast_n) \longrightarrow \alpha (Y_\infty) \longrightarrow 0$, where $\alpha(\ldotp)$ denotes the Iwasawa adjoint (with additional action by $G$). In particular, 
$(X^0_\infty)^\ast = ({\displaystyle\mathop{\lim_\leftarrow}}\, B_n^\ast)^0$. Let us take Fitting ideals in this exact sequence. A well known lemma (generally attributed to Cornacchia ; see e.g. \cite{NN}, lemma 3.4.2) states that for any torsion $\Lambda [\chi]$-module $M$, 
$Fitt_{\Lambda [\chi]} (M) = Fit_{\Lambda [\chi]} (M^0) \ldotp char_{\Lambda [\chi]} (M)$ where 
$char_{\Lambda [\chi]} (\ldotp)$ denotes the characteristic ideal. In the semi-simple situation, we can put the $\chi$-parts together to get~: 
$Fit_{\Bbb A} ({\displaystyle\mathop{\lim_\leftarrow}}\, B^\ast_n) = Fit_{\Bbb A}(X^0_\infty) \ldotp char_{\Bbb A} (\alpha (Y_\infty))$ (with an obvious definition of $char_{\Bbb A} (\ldotp)$ here). Let $(M)^{\#}$ be the module $M$ with inverted 
Galois action. It is classically known that $\alpha (M)$ is pseudo-isomorphic to $(M)^{\#}$, and that 
$Fit_{\Bbb A} (M^\ast)=Fit_{\Bbb A}(M)^{\#}$ since the $p$-Sylow subgroups of the $G'_n$'s are cyclic ([MW], appendix). 
Hence ${\cal D }_\infty = Fit_{\Bbb A} ({\displaystyle\mathop{\lim_\leftarrow}}\ B^\ast_n) = {\displaystyle\mathop{Fit_{\Bbb A}}(X^0_\infty)^{\#}} \ldotp {\displaystyle\mathop{char_{\Bbb A}(Y_\infty)}^{\#}}$. As $char_{\Bbb A} (Y_\infty) = char_{\Bbb A} (X_\infty)$ in the ``Gras type'' formulation of the WMC, we get~: $({\displaystyle\mathop{{\cal D }_\infty}})^{\#} = Fit_{\Bbb A} (X^0_\infty) \ldotp char_{\Bbb A} (X_\infty) = 
Fit_{\Bbb A} (X_\infty)$.\end{proof}

{\it Remarks~:}
\begin{enumerate}
 \item In spite of the presence of the algebra ${\Bbb A},$ proposition \ref{2.2} is not a genuine equivariant result. In particular, the definition 
of the characteristic ideal $char_{\Bbb A}(\ldotp)$ cannot be generalized to the non semi-simple case.
\item The ideal ${\cal D }(F_\infty)$ can easily be explicited in kummerian terms using $\S$  1.2. It suffices to start from the base field $E = F (\zeta_p)$ and then use (co) descent from $E$ to $F,$ or from $E_\infty$ to $F_\infty,$ which works smoothly because $p$ does not divide $[ E\, :\, F].$
\item It is well known that $\overline U'_\infty = H^1_{I_w,S_p} (F_\infty, {\Bbb Z}_p(1))$ in full generality. In the situation of proposition  
\ref{2.2}, we have also $X_\infty = X'_\infty = H^2_{I_w, S_p} (F_\infty, {\Bbb Z}_p (1))$ (for details, see proposition \ref{3.2} below). We can also consider the modules $H^2_{I_w, S_p} (F_\infty, {\Bbb Z}_p (m)),$ $m$ odd, as in the introduction and describe explicitly their ${\Bbb A}$-Fitting ideals by taking Tate twists above $E_\infty,$ and then doing (co)descent as in 2) (for details, $\S$ 3.4.3 below).
\end{enumerate}
\section{Equivariant study of the non semi-simple case}
In this case, as we noticed before, two major difficulties are encountered right from the start~: the notion of characteristic ideals of torsion modules 
(with appropriate multiplicative properties) is no longer available ; neither is the ``Gras type'' formulation of the WMC. 
The solution to both problems will come from the equivariant Iwasawa theory of complexes. Among many existing formulations, that of M. Witte 
(\cite{W}) seems the best suited to our purpose. Let us recall the minimal amount of definitions and results that we need, 
referring to \cite{W} for further details or greater generality.
\subsection{Perfect torsion complexes and characteristic ideals} (see \cite{W}, $\S$ 1)
Let $R$ be a commutative noetherian ring and $Q(R)$ its total ring of fractions. A complex $C^\cdot$ of $R$-modules is called a {\it torsion} complex if $Q(R) {\displaystyle\mathop\otimes_R}\ C^\cdot$ is acyclic; {\it perfect} if $C^\cdot$ is quasi-isomorphic to a bounded complex of projective $R$-modules. To a perfect complex $C^\cdot$ one can attach its Knudsen-Mumford determinant, denoted $det_R\, C^\cdot$. If moreover $C^\cdot$ is torsion, the natural isomorphism $Q(R) \simeq Q(R) {\displaystyle\mathop\otimes_R}\ det_R\, C^\cdot$ allows to see $det_R\, C^\cdot$ as an invertible fractional ideal of $R$. Recall that these invertible fractional ideals form a group ${\cal J }(R),$ which is isomorphic to $Q(R)^\ast/R^\ast$ if the ring $R$ is semi-local.
\begin{defi}\label{3.1}
The characteristic ideal of a perfect torsion complex $C^\cdot$ is \linebreak $char_R (C^\cdot) = (det_R\, C^\cdot)^{-1} \in {\cal J }(R)$.
\end{defi}
{\it Examples~:}
If $R$ is a noetherian and normal domain and $M$ is a torsion module which is perfect (considered as a complex concentrated in degree 0), i.e. of finite projective dimension, then $char_R (M)$ coincides with the ``content'' of $M$ in the sense of Bourbaki. If $R = \Lambda = {\Bbb Z}_p [[T]],$ then $char_\Lambda (M)$ is the usual characteristic ideal. This justifies the name of equivariant characteristic ideal for $char_{\Bbb A} (M)$.

\vspace{.3cm}

Many functorial properties of $char_R(\ldotp)$ are gathered in \cite{W}, proposition 1.5. We are particularly interested in the following~:

If $C^\cdot$ is a perfect torsion complex of $R$-modules such that cohomology modules of $C^\cdot$ are themselves perfect, then $char_R (C^\cdot) = \displaystyle\mathop\prod_{n \in {\Bbb Z}} (char_R\, H^n C^\cdot)^{(-1)^n}$.
\subsection{Iwasawa cohomology complexes}
For the rest of the paper, unless otherwise specified, $F/{\Bbb Q}$ will be an {\it abelian} number field (not necessarily totally real). With the notations and conventions of the beginning section 1, let us extract from \cite{W}, $\S$ 3 some perfect complexes and cohomology modules for our use (the situation in \cite{W} is more general)~:

The Iwasawa complex of ${\Bbb Z}_p(m)$ relative to $S$ is the cochain complex of continuous \'etale cohomology ${\bf R}\Gamma_{Iw} (F_\infty, {\Bbb Z}_p (m))$ as constructed by U. Jannsen. This is a perfect ${\Bbb A}$-complex whose cohomology modules are
$H^i_{Iw} (F_\infty, {\Bbb Z}_p(m))=H^i_{Iw, S} (F_\infty, {\Bbb Z}_p(m)) = {\displaystyle\lim_{\underset{\text{cores}}{\leftarrow}}} \ H^i_S (F_n, {\Bbb Z}_p (m))$ for $i = 1, 2,$ zero otherwise.

Let us gather in an overall proposition many known properties of these cohomology groups. Our main reference will be \cite{KNF}, sections 1 and 2 
(in which $S = S_p,$ but the proofs remain valid for any finite set $S$ containing $S_p)$. Let us fix again some notations~:

$E = F(\zeta_p),\ E_\infty = F(\zeta_{p^\infty}=,\ \Gamma^\times = Gal\, (E_\infty/F)$

$r_1$ (resp. $r_2$) = number of real (resp. complex) places of $F$

$\overline U_S(E) = U_S (E) \otimes {\Bbb Z}_p$ = the $p$-adic completion of the $S$-units of $E$

$\overline U_S (E_\infty) = {\displaystyle\mathop{\lim_\leftarrow}} \, \overline U_S(E_n)$ w.r.t. norms, $X'(E_\infty)$ = the totally split Iwasawa module above $E_\infty.$
\begin{prop}\label{3.2}
\ 

\begin{itemize}
\item[$(i)$] For any $m \in {\Bbb Z},$ $rank_{{\Bbb Z}_p}\, H^1_S (F, {\Bbb Z}_p(m)) = d_m + \delta_m,$ where $d_m = r_1 + r_2$ (resp. $r_2)$ if $m$ is odd (resp. even), and $\delta_m = rank_{{\Bbb Z}_p} (X'(E_\infty)(m-1)^{\Gamma^\times})$.
\item[$(ii)$] For any $m \in {\Bbb Z},$ $m \not= 0$, $tor_{{\Bbb Z}_p} H^1_S (F, {\Bbb Z}_p (m)) = H^0_S (F, {\Bbb Q}_p/{\Bbb Z}_p(m))$. 
In particular, if $m$ is odd and $F$ totally real, $H^1_S (F, {\Bbb Z}_p(m))$ is ${\Bbb Z}_p$-free.
\item[$(iii)$] For any $m \in {\Bbb Z}$, $m \not= 1$, there is a natural codescent exact sequence~:
\[0 \to \bigl( \overline U_S (E_\infty) (m-1)/\Lambda\text{\rm -tors}\bigl)_{\Gamma^\times} \to H^1_S (F, {\Bbb Z}_p(m))/{\Bbb Z}_{p}\text{\rm -tors} \to X' (E_\infty) (m-1)^{\Gamma^\times} \to 0\]
\item[$(iv)$] For any $m \in {\Bbb Z}$, $m \not= 1$, the Poitou-Tate sequence for $H^2_S$ can be written as~: 
$0 \to X'(E_\infty) (m-1)_{\Gamma^\times} \to H^2_S (F, {\Bbb Z}_p(m)) \to \displaystyle\mathop\oplus_{v \in S_f} H^2(F_v, {\Bbb Z}_p (m)) \simeq \linebreak \displaystyle\mathop\oplus_{v \in S_f} H^0 (F_v, {\Bbb Q}_p/{\Bbb Z}_p\, (1-m))^\ast \to H^0_S (F, {\Bbb Q}_p/{\Bbb Z}_p (1-m))^\ast \to 0$

For $m = 1,$ the leftmost term must be replaced by $A_S (F),$ the $p$-part of the $S_f$-class group of $F.$
\end{itemize}
\end{prop}

{\it Remark~:}
It is conjectured that $X'(E_\infty) (m-1)^{\Gamma^\times}$ is finite (i.e. $\delta_m = 0)$ for all $m \in {\Bbb Z}$ (see e.g. \cite{KNF}, p. 637). These are the so-called {\it $m^{th}$-twists of Leopoldt's conjecture}. The case $m = 0$ (resp. 1) corresponds to Leopoldt's (resp. Gross') conjecture. For $m \geq 2,$ $\delta_m = 0$ because $K_{2m-{2}}\, {\cal O }_F$ is finite. Recall that we implicitly suppose that $\delta_m = 0$ for all $m \in {\Bbb Z}.$

\vspace{.3cm}

At the infinite level, we have the following
\begin{lem}\label{3.3}
The ${\Bbb A}$-modules $H^i_{Iw} (F_\infty, {\Bbb Z}_p (m)),$ $i = 1, 2,$ are perfect.
\end{lem}
\begin{proof}
Let ${\Bbb Q}_\infty$ be the cyclotomic ${\Bbb Z}_p$-extension of ${\Bbb Q}$ and $\Lambda = {\Bbb Z}_p [[Gal\, ({\Bbb Q}_\infty/{\Bbb Q})]]$. 
Since $\Lambda$ is a regular noetherian ring, every $\Lambda$-noetherian module is perfect. For the cohomology modules over ${\Bbb A}$, 
we just have to use the following quasi-isomorphism established e.g. by 
\cite{FK1,Ne} etc.~: the natural ring homomorphism $\Lambda \to {\Bbb A}$ induces a quasi-isomorphism 
${\Bbb A} \displaystyle\mathop\otimes_\Lambda \displaystyle\mathop {\bf R}\, \Gamma_{Iw} ({\Bbb Q}_\infty, {\Bbb Z}_p(m)) 
\simeq \displaystyle\mathop{\bf R}\, \Gamma_{Iw} (F_\infty, {\Bbb Z}_p (m))$. \end{proof}

\subsection{The limit theorem}
As we explained in the introduction, the ``limit theorem'' of Burns and Greither (\cite{BG1}, thm 6.1) is actually an EMC, expressed in the language
of the Iwasawa theory of complexes, which encapsulates equivariant generalizations of both the WMC (for the minus part of $X_\infty$)
and its formulation ``\`a la Gras'' (for the plus part). It is ultimately derived from the WMC, but only after some hard work. 
We recall here its presentation by M. Witte (only for the characters $\kappa^m_{cyc}$; the characters in \cite{W} are more general).
Our main reference will be \cite{W}, sections 6 and 7. The ``limit theorem'' will relate a ``zeta-element'' and a ``special cyclotomic element'' constructed 
as follows~:

Let $f = N\, p^d$, $p \mid\!\!\!\!/ \, N$, the conductor of $F$. Introduce $L_\infty = {\Bbb Q} (\zeta_{N p^\infty}) \supset F_\infty$. 
At the $n^{th}$-levels, define the Stickelberg elements~:
\[Stick_n = \sum_{0 < a < f p^n \atop (a, fp) = 1} \Bigl( {a \over {f p^n}} - {1 \over 2} \Bigl) \Bigl( \displaystyle\mathop\prod_{\ell\ prime} Frob_\ell^{-v_\ell (a)}\Bigl) \in {\Bbb Q}_p [Gal\, (L_n/{\Bbb Q})]\ ,\]
and at the infinite level~: 
$Stick_\infty = \displaystyle\mathop{\lim_\leftarrow} \, Stick_n \in Q \bigl( {\Bbb Z}_p [[Gal (L_\infty/{\Bbb Q})]] \bigl)$.

Denote by $pr_{\pm}$ the projectors onto the $(\pm 1)$-eigenspaces of complex conjugation and define the {\it zeta-element} 
${\cal L }_S (F_\infty, \kappa^{1-m}_{cyc})$ as the image by the natural map 
$Q\, {\Bbb Z}_p [[(Gal (L_\infty/{\Bbb Q}))]] \longrightarrow Q\, {\Bbb A} = Q\, {\Bbb Z}_p [[G_\infty]]$ of the element 
$Tw_{m-1} (pr_{+} - pr_{-} Stick_\infty)$, where $Tw_a$ is the Tate twist induced by $\sigma \mapsto \kappa^a_{cyc} (\sigma) \ldotp \sigma$.

{\it Example~:} (\cite{W}, proposition 6.3)
Let $F_\infty = {\Bbb Q}_\infty$ and $S = S_p$ and write $\omega$ for the Teichm\"uller character. Then the image by the natural map 
$Q\, {\Bbb Z}_p [[\Gamma]] \to {\Bbb Q}_p$ of ${\cal L }_S ({\Bbb Q}_\infty, \kappa^{1-m}_{cyc})$ is $L_p (1-m, \omega^m \ldotp rec^{-1})$, 
where $rec\ :\ Gal\, ({\Bbb Q}(\zeta_N)/{\Bbb Q} \to ({\Bbb Z}/N {\Bbb Z})^\ast$ is the isomorphism induced by $Frob^{-1}_\ell \mapsto \ell$, 
and $L_p (\ldotp)$ is the Kubota-Leopoldt $p$-adic $L$-function.

Define now the {\it special cyclotomic element} $\displaystyle{\boldsymbol{\eta}}(F_\infty, \kappa^m_{cyc})$ as the image by the \linebreak composite map 
$H^1_{Iw} (L_\infty, {\Bbb Z}_p(1)) = \overline U_S(L_\infty) \displaystyle\mathop{\longrightarrow}^{Tw_{m-1}}\, H^1_{Iw} (L_\infty, {\Bbb Z}_p (m)) = 
\overline U_S (L_\infty) (m-1) \linebreak \displaystyle\mathop{\longrightarrow}^{nat} H^1_{Iw} (F_\infty, {\Bbb Z}_p (m))$ of the element 
$pr_{+} ((1-\zeta_{N p^n})_{n \geq 0})$. The element $\boldsymbol{\eta} (F_\infty, \kappa^m_{cyc})$ allows to modify the Iwasawa complex 
$\displaystyle {\bf R}\Gamma_{Iw} (F_\infty, {\Bbb Z}_p(m))$ to get a perfect torsion complex. 
Precisely, there exists a unique morphism ${\Bbb A} \boldsymbol{\eta} (F_\infty, \kappa^m_{cyc}) [-1] \to \displaystyle{\bf R} \Gamma_{Iw} (F_\infty, {\Bbb Z}_p (m))$ in the derived category which induces the natural inclusion on cohomology, hence a unique (up to quasi-isomorphism) complex denoted by 
$\displaystyle{\bf R}\, \Gamma_{Iw}/ \boldsymbol{\eta} (F_\infty, \kappa^m_{cyc})$ which takes place in the following distinguished triangle
\[{\Bbb A}\, \boldsymbol{\eta} (F_\infty, \kappa^m_{cyc}) [-1] \to \mathop {\bf R}\, 
\Gamma_{Iw} (F_\infty, {\Bbb Z}_p(m)) \to   {\bf R}\, \Gamma_{Iw}/ \boldsymbol{\eta} (F_\infty, \kappa^m_{cyc})\ .\]

\begin{lem}\label{3.4} (\cite{W}, lemma 7.2)
\begin{itemize}
\item[$(i)$] $\displaystyle{\bf R}\Gamma_{Iw}/ \boldsymbol{\eta} (F_\infty, \kappa^m_{cyc})$ is a perfect torsion complex of ${\Bbb A}$-modules
\item[$(ii)$] If $F_\infty$ is totally real and $m$ is odd, then ${\Bbb A} \boldsymbol{\eta}(F_\infty, \kappa^m_{cyc}$ is a free ${\Bbb A}$-module of rank 1.
\end{itemize}
\end{lem}
We can now state the ``limit theorem'' (\cite{BG1}, theroem 6.1; \cite{W}, theorem 7.4).
\begin{theo}\label{3.5}
For an abelian number field $F/{\Bbb Q}$, with $S = S_\infty \cup S_p \cup \, Ram (F/{\mathbb {Q}})$, for any $m\in \mathbb{Z}$~:
\begin{itemize}
 \item[$(i)$] At any prime ${\cal P }$ of codimension 1 of ${\Bbb A}$, containing $p$, the localized complex $\bigl( {\bf R}\, \Gamma_{Iw} / \boldsymbol{\eta}
 (F_\infty, \kappa^m_{cyc}) \bigl)_{\cal P }$ is acyclic (``vanishing of the $\mu$-invariant'').
\item[$(ii)$] ${\cal L }_S (F_\infty, \kappa^{1-m}_{cyc})$ generates the ${\Bbb A}$-characteristic ideal of ${\bf R}\, \Gamma_{Iw}/ \boldsymbol{\eta} (F_\infty, \kappa^m_{cyc})$.
\end{itemize}
\end{theo}
To stress the difference between the (equivariant) limit theorem above and results obtained character by character (such as in \cite{HK}), let us cite the following comparison lemma (\cite{W}, lemma 7.6)~:
\begin{lem}\label{3.5bis}
Suppose for simplification that $F$ is linearly disjoint from ${\Bbb Q}_\infty$ and $p^2 \nmid f$. 
Let $\phi \ :\ \Omega = {\cal O }_p [[G_\infty]] \to \widetilde\Omega = \displaystyle\mathop\Pi_{\chi \in \widehat G}\, {\cal O }_p [[Gal\, 
({\Bbb Q}_\infty/{\Bbb Q})]]$ be the normalisation of $\Omega$ in $Q\, \Omega$, where ${\cal O }_p$ is the ring obtained from 
${\Bbb Z}_p$ by adding all the values of all the characters of $G$. Then
\[char_{\widetilde \Omega} \bigl( {\bf L}\,\phi_\ast\,\mathop {\bf R}\, \Gamma_{Iw} / \boldsymbol{\eta} (F_\infty, {\cal O }_p (m)) = 
L_S (F_\infty, \kappa^{1-m}_{cyc} \bigl) \ldotp \widetilde \Omega\]
\end{lem}

\subsection{The determinant of $H^2_{Iw} (F_\infty, {\Bbb Z}_p (m))$}
From now on, $F_\infty$ is {\it totally real} and $m \in {\Bbb Z}$ is {\it odd}. For $m < 0$, 
we assume implicitly the validity of the ``twisted'' Leopoldt conjectures (for $m > 1$, this is a theorem, see proposition \ref{3.2}(i)). 
To state and prove our main result, we shall proceed in several steps.

Because the cohomology ${\Bbb A}$-modules $H^i_{Iw} (\ldotp)$ are perfect (lemma \ref{3.3}), it follows from thm. \ref{3.5}(ii) and lemma \ref{3.3} 
that 
\[\begin{aligned} char_{\Bbb A}\bigl( H^1_{Iw} (F_\infty, {\Bbb Z}_p(m))/{\mathbb{A}} \boldsymbol{\eta} (F_\infty, \kappa^m_{cyc})\bigl)^{-1} \ldotp 
char_{\Bbb A}\, H^2_{Iw} (F_\infty, {\Bbb Z}_p(m)) & = ({\cal L }_S (F_\infty, \kappa^{1-m}_{cyc}) \bigl)\\ &= (1)\ , \end{aligned}\] 
so that it remains only to determine the first equivariant characteristic series, appealing to the algebraic results of section 1.

\subsubsection{} The point is to choose the admissible pair $(V_n, W_n)_{n \geq 0}$ attached to $F_\infty/F.$\hfill {\it Fix an odd interger $m \not= 1$} and for any $n \geq 0$, write $\eta_n^{(m)}$ for the image of $\boldsymbol{\eta} (F_\infty, \kappa^m_{cyc})$ 
by the natural map $H^1_{Iw} (F_\infty, {\Bbb Z}_p(m)) \to H^1_S (F_n, {\Bbb Z}_p(m))$. 
Then {\it $\eta^{(m)}_n$ is ${\Bbb Z}_p [G_n]$-free}~: this comes from lemma \ref{3.4} (ii) and the codescent exact sequence of propos. \ref{3.2} (i); for a direct argument, see \cite{BB}, thm. 3.4. Note that highly non trivial ingredients are needed for both proofs~: Bloch-Kato's reciprocity law for the first, twisted Leopoldt's conjecture for the second. We shall take $V_n = H^1_S (F_n, {\Bbb Z}_p(m))$ and $W_n = {\Bbb Z}_p [G_n] \ldotp \eta_n^{(m)}$. Adding a superscript $(\ldotp)^{(m)}$ 
to the notations of section 1, let ${\cal D }^{(m)} (F_\infty)$ = ${\rm Fit}_{\Bbb A}\, (\displaystyle\mathop{\lim_{\leftarrow}} B^\ast_n)$, 
which we must relate to the desired ${\Bbb A}$-determinant.

{\it Remark~:} The reason for choosing a twist $m \not= 1$ is that codescent on $S$-units (corresponding to $m = 1$) 
is notoriously not smooth, especially in case of $p$-decomposition. We shall reintegrate $m = 1$ in thm. \ref{3.7} by using a ``twisting trick''.

Since the $p$-Sylow subgroups of the $G_n$'s are no longer necessarily cyclic, the argument on Fitting ideals used in the semi-simple case (proposition \ref{2.2}) no longer works. We must consequently change the definition of $V_n = H^1_S (F_n, {\Bbb Z}_p (m))$, replacing it by $V'_n =$ image of 
$\bigl(  H^1_{Iw} (F_\infty, {\Bbb Z}_p(m)) \to H^1_S (F_n, {\Bbb Z}_p (m)) \bigl)$. We don't change 
$W_n = {\Bbb Z}_p [G_n] \ldotp \eta_n^{(m)}$ and we define $B'_n = V'_n/W_n$. By proposition
 \ref{3.2} (iii), $\displaystyle\mathop{\lim_\leftarrow} V'_n = \displaystyle\mathop{\lim_\leftarrow} V_n = H^1_{Iw} (F_\infty, {\Bbb Z}_p (m))$, hence $\displaystyle\mathop{\lim_\leftarrow} B'_n = \displaystyle\mathop{\lim_\leftarrow} B_n := B_\infty = H^1_{Iw} (F_\infty, {\Bbb Z}_p(m))/{\Bbb A} \boldsymbol{\eta} (F_\infty, \kappa^m_{cyc}),$ whereas $\displaystyle\mathop{\lim_\leftarrow} B_n^{' \ast} = \alpha (\displaystyle\mathop{\lim_\leftarrow} B'_n) = \alpha (B_\infty)$, $\alpha(\ldotp)$ denoting the Iwasawa adjoint, with 
additional action by $H = Gal (F_\infty/k_\infty)$. Note that this adjoint module is naturally isomorphic to ${\rm Ext}_{\Bbb A} (B_\infty, {\Bbb A})$ (see e.g. \cite{Ng}, $\S$ 3) over ${\Bbb A}$. 
Over $\Lambda$, we know that $\alpha(B_\infty)$ and $({\displaystyle\mathop{B_\infty}})^{\#}$ are pseudo-isomorphic, hence the existence of an exact sequence 
(non canonical) of $\Lambda$-modules $0 \to \alpha(B_\infty) \to ({\displaystyle\mathop{B_\infty}})^{\#} \to \Phi \to 0,$ where $\Phi$ is a finite abelian $p$-group. Since $H$ acts on the first two terms, it also acts on the third, i.e. the above sequence is indeed exact over ${\Bbb A}$.

\subsubsection{} Since $\displaystyle\mathop{\lim_\leftarrow} V'_n = \displaystyle\mathop{\lim_\leftarrow} V_n$,
the same reasoning exactly as in proposition \ref{1.1} shows that 
${\rm Fit}_{\Bbb A} (\displaystyle\mathop{\lim_\leftarrow} B_n^{' \ast}) = {\rm Fit}_{\Bbb A} 
(\alpha (B_\infty)) = {\cal D }^{(m)} (F_\infty)$ 
(the same $\cal {D} ^{(m)}(F_\infty)$ as before). In particular, this Fitting ideal is monogeneous according to proposition \ref{1.3} and an obvious descent from $E_\infty = F_\infty (\mu_{p^\infty})$ to $F_\infty$.
 It remains to compare the two monogeneous ideals $I = {\rm Fit}_{\Bbb A} (\alpha (B_\infty))$ and 
$J = char_{\Bbb A} ({\displaystyle\mathop{B_\infty}})^{\#}$ by using localization~:
\begin{lem}\label{3.6} (\cite{BG1}, lemma 6.1)
Let $R$ be a Cohen-Macaulay ring of dimension 2 and let $I, J$ be two invertible fractional ideals of $R$. 
Then $I = J$ if and only if $I_{\cal P } = J_{\cal P }$ for all height one prime ideals ${\cal P }$ of $R$.
\end{lem}
Cutting out if necessary by the characters of the non-$p$-part of $H,$ we can suppose that our ring ${\Bbb A}$ is like in lemma \ref{3.6} and proceed to localization~:

- at a height one prime $\mathcal{P}$ not containing $p,$ ${\Bbb A}_{\mathcal{P}}$ is a discrete valuation ring in which $p$ is invertible. It follows that ${\rm Fit}\, (\ldotp)$ and ${\rm det}\, (\ldotp)^{-1}$ coincide over ${\Bbb A}_{\mathcal {P}},$ and $\Phi_{\cal P } = (0),$ hence $I_{\cal P } = J_{\cal P }$.

- at the unique height one prime ${\cal P }$ of ${\Bbb A}$ containing $p$, the vanishing of the $\mu$-invariant (thm \ref{3.5} (i)) means that ${\displaystyle(\mathop{B_\infty})_{\cal P }^{\#}}$ vanishes, hence also $\alpha (B_\infty)_{\cal P }$.

We have thus shown that $char_{\Bbb A}({\displaystyle\mathop{B_\infty}})^{\#} = {\rm Fit}_{\Bbb A} (\alpha(B_\infty)) = 
{\mathcal {D} }^{(m)}(F_\infty)$. We can now state and prove our main result~:
\begin{theo}\label{3.7} 
Let $F/{\Bbb Q}$ be a totally real abelian number field, $E = F(\zeta_p)$, ${\Bbb A} = {\Bbb Z}_p [[\Gal(F_\infty/{\Bbb Q})]]$, 
${\Bbb B} = {\Bbb Z}_p [[\Gal(E_\infty/{\Bbb Q})]]$. Then, for any odd $m \in {\Bbb Z}$, we have~:
\begin{itemize}
 \item[$(i)$] ${\rm char}_{\Bbb B} (H^2_{Iw} (E_\infty, {\Bbb Z}_p (m)) = 
{\displaystyle\mathop{{\cal D }^{(1)} (E_{\infty})^{\#} (m-1)} = tw_{m-1} {\cal D }^{(1)} (E_\infty)}$, 
where $tw_{m-1}$ is the Iwasawa twist induced by $\sigma  \mapsto \kappa^{m-1}_{cyc} (\sigma) \sigma^{-1}$.
\item[$(ii)$] ${\rm char}_{\Bbb A} (H^2_{Iw} (F_\infty, {\Bbb Z}_p(m)) = 
({\displaystyle\mathop{e_{m-1} {\cal D }^{(1)} (E_{\infty})}})^{\#}$, where $e_{m-1}$ is the idempotent of $\Delta$ 
associated to the power $\omega^{m-1}$ of a generator $\omega$ of $\widehat \Delta.$
\end{itemize}
\end{theo}
\begin{proof}
For any odd $m \not= 1,$ we have just seen that 
\[{\displaystyle\mathop{{\rm char}}_{\Bbb A} (H^2_{Iw} (F_\infty, {\Bbb Z}_p(m)))^{\#} = 
{\cal D }^{(m)} (F_\infty) = \nu {\cal D }^{(m)} (E_\infty)}\ ,\]where $\nu$ denotes the norm map of $\Delta$. But ${\cal D }^{(m)} (E_\infty)={\cal D }^{(1)} (E_\infty)(m-1)$, so that    $\nu {\cal D }^{(m)} (E_\infty)= e_{m-1} {\cal D }^{(1)} (E_\infty)$, because $\Delta$ is of order prime to $p$.
Besides, denoting by $\pi$ the natural projection ${\Bbb B} \to {\Bbb A},$ we know that 
$\displaystyle\mathop {\bf L} \pi_\ast \displaystyle\mathop {\bf R}\, \Gamma_{Iw}(E_\infty,\mathbb{Z}_p(m))$ is naturally quasi-isomorphic to 
$\displaystyle\mathop{\bf  R}\, \Gamma_{Iw} (F_\infty, {\Bbb Z}_p (m))$ (\cite{W}, propos. 3.6 (ii)). Hence, by definition of the determinant, ${\rm char}_{\Bbb A} (H^2_{Iw} (F_\infty, {\Bbb Z}_p (m)) = \nu\, (char_{\Bbb B}\, H^2_{Iw} (E_\infty, {\Bbb Z}_p (m))) = e_{m-1} (char_{\Bbb B}\, H^2_{Iw} (E_\infty, {\Bbb Z}_p(1))).$ Since the idempotent $e_{m-1}$ depends only on the 
residue class 
of 
$(m-1)$ mod $(p-1),$ we can conclude that the latter equality is valid for all odd $m \in {\Bbb Z}$, so that 
${\rm char}_{\Bbb B}\, H^2_{Iw} (E_\infty, {\Bbb Z}_p(1)) = 
{\displaystyle\mathop{{\cal D }^{(1)}} (E_{\infty})^{\#}}$, and 
${\rm char}_{\Bbb B} H^2_{Iw} (E_\infty, {\Bbb Z}_p (m)) ={\displaystyle\mathop{{\cal D }^{(1)}} (E_{\infty})^{\#}} (m-1) = tw_{m-1} {\cal D }^{(1)} (E_\infty)$. 
\end{proof}
Note that in the semi-simple case, thm. \ref{3.7} above contains thm. 5.3.2 of \cite{NN}.
\subsection{The Fitting ideal of $H^2_{Iw} (F_\infty, {\Bbb Z}_p(m))$}
The next natural step would be to perform (co)descent on thm.  \ref{3.7}. Because $c\, d_p\, G_S (F) \leq 2$, 
the map \linebreak $H^2_{Iw} (F_\infty, {\Bbb Z}_p(m))_{\Gamma^{p^n}} \to H^2_S (F_n, {\Bbb Z}_p (m))$ is an isomorphism, but the latter module needs no longer to be perfect 
(or, equivalently, cohomologically trivial) over ${\Bbb Z}_p [G_n]$. Actually, in the exact sequence of propos. \ref{3.2} (iv), our knowledge of the cohomology of the codescent module $X' (E_\infty) (m-1)_{\Gamma^{\times  p^n}}$ is $\ldots$ less than perfect. A way to turn the difficulty would be to replace determinants by Fitting ideals, which are compatible with codescent. But for this we need an equivariant analogue of Cornacchia's lemma which was used in the proof of proposition \ref{2.2}.
To this end, we would like to add a technical condition which the reader would rightly find too brutal if it were imposed {\it ex abrupto} without any explanation. 
Hence the following preliminaries~:
\subsubsection{} Let us recall briefly the whereabouts of Cornacchia's lemma~: for a noetherian $\Lambda$-module $M$, 
denote by $M^0$ its maximal finite submodule and write 
$\widetilde M = M/M^0$. Since the global projective dimension of $\Lambda$ is equal to 2, $pd_\Lambda\, \widetilde M \leq 1$ because 
$\widetilde M^0 = (0)$ (Auslander-Buchsbaum), hence ${\rm Fit}_\Lambda (M) = {\rm Fit}_\Lambda (M^0) \ldotp {\rm Fit}_\Lambda (\widetilde M)$. 
If moreover $M$ is $\Lambda$-torsion, ${\rm Fit}_\Lambda(\widetilde M) = {\rm char}_\Lambda (\widetilde M) = {\rm char}_\Lambda (M)$. 
The difficulty when passing from $\Lambda$ to ${\Bbb A}$ is that the Auslander-Buchbaum result is no longer available. 
We shall use instead a weak substitute due to Greither. Recall that ${\Bbb A} = \Lambda [H]$, where $H = Gal\, (F_\infty/k_\infty)$.
\begin{lem}\label{3.8} (\cite{G}, propos. 2.4)
If an ${\Bbb A}$-noetherian torsion module $N$ is cohomologically trivial over $H$ and has no non-trivial finite submodule, then $pd_{\Bbb A} (N) \leq 1$.
\end{lem}
In the sequel, we shall work over $E_\infty = F (\zeta_{p^\infty})$, fix an odd $m \in {\Bbb Z}$, $ m \not= 1$, 
and consider the module $M := H^2_{Iw} (E_\infty, {\Bbb Z}_p (m))$. 
According to proposition \ref{3.2} and after taking ${\displaystyle\mathop{\lim_{\longleftarrow}}}$ with respect to corestriction maps, we have an exact sequence~:

$0 \to X'(E_\infty) (m-1) \to M \to W_m (E_\infty) \to 0$, where $W_m (E_\infty)$ is defined tautologically and will be given an explicit description below. 
This shows in particular that $M^0 = X'(E_\infty)^0 (m-1)$. We want to get hold of $\widetilde M = M/M^0$. 
Applying the snake lemma to the commutative diagram~:
\[\xymatrix{
0\ar[r]  & X'(E_\infty) (m-1)^0 \ar[r]\ar@{^(->} [d] & \ M \ar@{=}[d]\ar[r]& \widetilde M\ar[d]\ar[r] & 0\\
0\ar[r]  & X'(E_\infty) (m-1)\ar[r] & M\ar[r] & W_m (E_\infty)\ar[r] & 0\ \ ,\\
}\]
we get an exact sequence $0 \to \widetilde{X'(E_\infty)} (m-1) \to \widetilde M \to W_m (E_\infty) \to 0$. 
We intend to apply Greither's lemma to the rightmost module $W_m (E_\infty)$, which is the inverse limit of the kernels 
$W_m (E_n) = \, Ker\, \Bigl( \displaystyle\mathop\oplus_{v \in S_f} H^2 (E_{n, v}, {\Bbb Z}_p (m)) \simeq \displaystyle\mathop\oplus_{v \in S_f} H^0 (E_{n, v}, 
{\Bbb Q}_p/{\Bbb Z}_p(1-m))^\ast \to H^0 (E_n, {\Bbb Q}_p /{\Bbb Z}_p (1-m))^\ast \Bigl)$. 
The limit of the semi-local modules is \linebreak $\displaystyle\mathop\oplus_{v \in S_f} H^2_{Iw} (E_{\infty, v}, {\Bbb Z}_p(m))$, 
which is perfect (for the same reason as in lemma \ref{3.3}), hence $H$-cohomologically trivial. As for the limit of the global 
$H^0(\ldotp)^\ast$'s, note that the transition maps reduce to $p^a$-power maps between twisted roots of unity; since 
$H^i (H, \displaystyle\mathop{\lim_\leftarrow}) = \displaystyle\mathop{\lim_\leftarrow} H^i(H, \ldotp)$ and $H$ is finite, we readily get the 
$H$-cohomological triviality of the limit.

Summarizing, lemma \ref{3.8} applies and $W_m (E_\infty)$ has projective dimension $\leq 1$. Going down to $F_\infty,$ we get an exact sequence 
$0 \to e_{m-1} \widetilde{X'} (E_\infty) \to H^2_{Iw} (F_\infty, {\Bbb Z}_p (m)) \to W_m (F_\infty) \to 0$, 
with $pd_{\Bbb A} W_m (F_\infty) \leq 1$. Since $m$ is odd, the leftmost module lives in the ``plus'' part. 
Let $E_\infty^{(m-1)}$ be the subfield of $E_\infty$ cut out by the character $\omega^{m-1}$ (notations of thm \ref{3.7}).

{\it Assume Greenberg's conjecture for} $E_\infty^{(m-1)}/E^{(m-1)}$, which is equivalent to the vanishing of $e_{m-1} \widetilde{X'} (E_\infty)$ 
and yields an exact sequence~:
\[0 \to e_{m-1} X'(E_\infty)^0 \to H^2_{Iw} (F_\infty, {\Bbb Z}_p (m)) \to W_m (F_\infty) = \widetilde{H^2_{Iw}(F_\infty, {\Bbb Z}_p (m))} \to 0\ \ \eqno(3)\]
Since $pd_{\Bbb A}\, W_m (F_\infty) \leq 1$, the Fitting ideal behaves multiplicatively and we have~:
 ${\rm Fit}_{\Bbb A}\, H^2_{Iw} (F_\infty, {\Bbb Z}_p (m))={\rm Fit}_{\Bbb A} \bigl( e_{m-1} X'(E_\infty)^0 \bigl)\ldotp {\rm Fit}_{\Bbb A}\, W_m (F_\infty)$. 
Moreover, \linebreak ${\rm Fit}_{\Bbb A}\, W_m (F_\infty)={\rm char}_{\Bbb A}\, W_m (F_\infty) =\, {\rm char}_{\Bbb A}\, 
\displaystyle{\widetilde{H^2_{Iw} (F_\infty, {\Bbb Z}_p(m))}} =
\, {\rm char}_{\Bbb A}\, H^2_{Iw} (F_\infty, {\Bbb Z}_p(m))$ (the last equality is obtained by the already used localization argument). Finally~:

\begin{theo}\label{3.9}
For any odd $m \in {\Bbb Z},$ $m \not= 1,$ assume Greenberg's conjecture for $E_\infty^{(m-1)}/E^{(m-1)}.$ Then~:

\begin{itemize}
 \item[$(i)$] At infinite level, \[{\rm Fit}_{\Bbb A}\, H^2_{Iw} (F_\infty, {\Bbb Z}_p(m)) =\, {\rm Fit}_{\Bbb A}(e_{m-1}\, X'(E_\infty)^0) \ldotp 
(e_{m-1}\, {\cal D }^{(1)} (E_\infty))^{\#}\ .\]
\item[$(ii)$] At any finite level $n \geq 0$, \[{\rm Fit}_{{\Bbb Z}_p[G_n]}\, H^2_S (F_n, {\Bbb Z}_p (m)) =\, {\rm Fit}_{{\Bbb Z}_p[G_n]}\, 
(e_{m-1}\, X'(E_\infty)^0)_{\Gamma^{p^n}} \ldotp \hfill (e_{m-1}\, {\cal D }_n (E))^{\#}\ .\]
\end{itemize}
\end{theo}
{\it Remark~:}
The perfectness of the two last terms in the exact sequence (3) implies that of the first term, hence the existence of 
${\rm char}_{\Bbb A} (e_{m-1}\, X'(E_\infty)^0)$. 
But the usual localization argument shows that this ideal is (1), and we recover a particular case of thm. \ref{3.7} (ii).

\subsubsection{} Recall that the module ${\cal D }_n (E)$ was defined in $\S$ 1.2 and its quotient ${\rm mod}\, p^n$ was described explicitly in kummerian terms. The interest of thm. \ref{3.9} (ii) lies in its comparison with already known results on refinements (for $m$ odd) of the Coates-Sinnott conjecture on Galois annihilators of the modules $H^2_S (F_n, {\Bbb Z}_p (m)) \simeq K_{2m-2} ({\cal O }_{F_n} [1/S])\otimes \mathbb{Z}_p$ (by Quillen-Lichtenbaum's conjecture, now a theorem). Note that for $m$ odd, the usual formulation of Coates-Sinnott gives no information other than ``zero kills everybody''. A refined conjecture was formulated by \cite{Sn} (resp. \cite{Nick}) in terms of leading terms (rather than values) of Artin $L$-functions at negative integers for abelian (resp. general) Galois extensions of number fields, and shown to be a consequence of the equivariant Tamagawa number conjecture (ETNC) for the Tate motives attached to these extensions. Let us return to the situation of the introduction, 
where $F/k$ is an 
abelian 
extension with group $G$, $k$ is totally real and $F$ is $CM$. We need to recall quickly the construction of a 
``canonical fractional ideal'' by Snaith and Nickel. We follow the presentation of \cite{Nick}, adapting it to our situation~:

- for the algebraic part, fix $m \geq 2$ and let ${\cal H }_{1-m} (F) = \displaystyle\mathop\oplus_{S_\infty} (2 \pi i)^{m-1} {\Bbb Z}$,
with action of complex conjugation (diagonally on $S_\infty$ and on $(2 \pi i)^{m-1}$). The Borel regulator 
$\rho_{1-m} : K_{2m-1} ({\cal O }_{F})\otimes \mathbb{Q} \to {\cal H }^{+}_{1-m} (F \otimes \mathbb{R})$ yields the existence of a
${\Bbb Q} [G]$-isomorphism $\phi_{1-m}\ :\ {\cal H }_{1-m} (F)^{+} \otimes {\Bbb Q} \displaystyle\mathop\to^\sim K_{2m-1} ({\cal O }_F) \otimes {\Bbb Q}$.
This allows, by applying the Quillen-Lichtenbaum conjecture (now a theorem), to construct a $G$-equivariant embedding (we fix $m$ and drop the index) 
$\phi\ :\ {\cal H }_{1-m} (F)^{+} \otimes {\Bbb Z}_p \hookrightarrow H^1_S (F, {\Bbb Z}_p (m))$ (note that this $H^1_S(\ldotp)$ does not depend on $S \supset S_p$, 
by a result of Soul\'e). The algebraic part of the canonical ideal will be ${\rm Fit}_{{\Bbb Z}_p [G]} ((coker\, \phi)^\ast)$.

- for the analytic part, consider the ring $R(G)$ of virtual characters of $G$ with values in $\overline{\Bbb Q}$ and let $L$ be a finite Galois extension of ${\Bbb Q}$ such that each representation of $G$ can be realized over $L$; let also $V_\chi$ be an $L[G]$-module with character $\chi$. We can form regulator maps~:
\[R_{\phi_{1-m}}\ :\ R(G) \longrightarrow {\Bbb C}^\ast\]
\[ \chi \mapsto\ det \Bigl( \rho_{1-m} \ldotp \phi_{1-m} \mid \ \Hom_G (V_{\displaystyle\mathop\chi^\vee}, {\cal H }^{+}_{1-m} (F) \otimes {\Bbb C} \bigl)\]
(where $\displaystyle\mathop\chi^\vee$ denotes the contragredient character).

Define then a function ${\cal A }^S_{\phi_{1-m}}\ :\ R(G) \to {\Bbb C}^\ast$, $\chi \mapsto R_{\phi_{1-m}} (\chi) / L^\ast_S (1-m, \chi)$, 
where $L_S (s, \chi)$ is the $S$-truncated Artin $L$-function and $L^\ast_S (1-m, \chi)$ is the leading term of this function at $(1-m)$. 
Gross' higher analogue of Stark's conjecture states that ${\cal A }^S_{\phi_{1-m}} (\chi^\sigma) = {\cal A }^S_{\phi_{1-m}} (\chi)^\sigma$ 
for all $\sigma \in Aut\,({\Bbb C})$. Assuming this conjecture and choosing an identification ${\Bbb C} \simeq {\Bbb C}_p,$ we can now define the 
``$p$-adic canonical ideal'' (we drop the index in $\phi_{1-m}$).
\begin{defi}\label{3.10} $\mathfrak{K}^S_{1-m} (p) = \ ({\rm Fit}_{{\Bbb Z}_p [G]} \bigl(coker\, \phi)^\ast \ldotp ({\cal A }^S_\phi)^{-1}\bigl))^{\#}$.
\end{defi}
NB~: this is actually an ``intermediate'' ideal on the construction of Snaith-Nickel, but it is all we need.

\begin{theo}\label{3.11} (\cite{Sn,Nick})
Suppose that Gross' conjecture, and also the ETNC for the pair $({\Bbb Q} (1-m)_F, {\Bbb Z} [{1 \over 2}] [G])$, $m \geq 2$, hold for the abelian extension $F/k$, 
where $k$ is totally real and $F$ is $CM$. Then ${\rm Fit}_{{\Bbb Z}_p [G]} H^2_S (F, {\Bbb Z}_p (m)) = \mathfrak{K}^S_{1-m} (p)$.
\end{theo}
\begin{proof} See \cite{Nick}, end of the proof of thm 4.1, as well as remark, p. 14.\end{proof} 
{\it Remarks and prospective~:}
\begin{enumerate}
 \item For even $m \geq 2$, further calculations show that thm \ref{3.11} contains the $p$-part of the usual Coates-Sinnott conjecture.
\item If $k = {\Bbb Q}$, Gross' conjecture and the ETNC hold true, and thm  \ref{3.11} becomes unconditional. Its comparison with thm  \ref{3.9} could give an analytic meaning to the parasite modules $(e_{m-1} X'(E_\infty)^0)_{\Gamma^{p^n}}$. ``Numerical'' information could also be obtained by computing the orders of the groups $K_{2n-2} ({\cal O }_{F_n})\otimes\mathbb{Z}_p$. 
For example, in the semi-simple case, where $\Gamma_n$ intervenes in place of $G_n$, 
this order was computed by \cite{M} on terms of values of $L_p$-functions at {\it positive} integers.
\item Instead of leading terms of Artin $L$-functions, one could also appeal to derivatives as in \cite{BdJG}. A natural (but resting only on thin air) query would be~: is there any conceptual link with thm  \ref{3.7}, knowing that ${\cal D }_n (E)$ can be interpreted as a module of ``$p$-adic Gauss sums'' (\cite{NN}, $\S$ 4.1) ?
\item A natural expectation would be the extension of thm  \ref{3.7} to the relative abelian case $(k \not= {\Bbb Q})$. But then a serious obstacle is 
the absence of special elements (at least non conjecturally). Partial progress has been made by \cite{Nico}, starting from the idea of replacing special elements 
by $L_p$-functions over $k$; this is a natural idea since, when $k = {\Bbb Q}$, special elements and $L_p$-functions are ``equivalent'' by Coleman's theory.
\item Finally, one would of course wish to deal with the non abelian case, in view of the non commutative EMC recently proved by \cite{Ka} and \cite{RW2}. But a non commutative analogue of the ``limit theorem'' is missing.
\end{enumerate}

\vspace{.3cm}

{\bf Acknowledgements~:} It is a pleasure for the author to thank Dr. Malte Witte for many useful discussions on his ``limit theorem''.


\backmatter
\bibliography{marsnguyen13}

\providecommand{\bysame}{\leavevmode ---\ }
\providecommand{\og}{``}
\providecommand{\fg}{''}
\providecommand{\smfandname}{et}
\providecommand{\smfedsname}{\'eds.}
\providecommand{\smfedname}{\'ed.}
\providecommand{\smfmastersthesisname}{M\'emoire}
\providecommand{\smfphdthesisname}{Th\`ese}
\begin{thebibliography}{KNQDF96}

\bibitem[BB12]{BB}
{\scshape T.~Beliaeva {\normalfont \smfandname} J.-R. Belliard} -- {\og Indices
  isotypiques des \'el\'ements cyclotomiques\fg}, \emph{Tokyo J. Math.}
  \textbf{35} (2012), no.~1, p.~139--164.

\bibitem[BdJG12]{BdJG}
{\scshape D.~Burns, R.~de~Jeu {\normalfont \smfandname} H.~Gangl} -- {\og On
  special elements in higher algebraic {$K$}-theory and the
  {L}ichtenbaum-{G}ross conjecture\fg}, \emph{Adv. Math.} \textbf{230} (2012),
  no.~3, p.~1502--1529.

\bibitem[BG03a]{BG1}
{\scshape D.~Burns {\normalfont \smfandname} C.~Greither} -- {\og On the
  equivariant {T}amagawa number conjecture for {T}ate motives\fg},
  \emph{Invent. Math.} \textbf{153} (2003), no.~2, p.~303--359.

\bibitem[BG03b]{BG2}
{\scshape D.~Burns {\normalfont \smfandname} C.~Greither} -- {\og Equivariant
  {W}eierstrass preparation and values of {$L$}-functions at negative
  integers\fg}, \emph{Doc. Math.} (2003), no.~Extra Vol., p.~157--185
  (electronic), Kazuya Kato's fiftieth birthday.

\bibitem[FK06]{FK1}
{\scshape T.~Fukaya {\normalfont \smfandname} K.~Kato} -- {\og A formulation of
  conjectures on {$p$}-adic zeta functions in noncommutative {I}wasawa
  theory\fg}, \emph{Proceedings of the {S}t. {P}etersburg {M}athematical
  {S}ociety. {V}ol. {XII}} (Providence, RI), Amer. Math. Soc. Transl. Ser. 2,
  vol. 219, Amer. Math. Soc., 2006, p.~1--85.

\bibitem[FK12]{FK2}
\bysame , {\og On conjectures of {S}harifi\fg}, \emph{Iwasawa 2012 conference},
  Heidelberg, July 30--August 3, 2012.

\bibitem[Fla11]{Fl}
{\scshape M.~Flach} -- {\og On the cyclotomic main conjecture for the prime
  2\fg}, \emph{J. Reine Angew. Math.} \textbf{661} (2011), p.~1--36.

\bibitem[GP11]{GP}
{\scshape C.~Greither {\normalfont \smfandname} C.~Popescu} -- {\og An
  equivariant main conjecture in iwasawa theory and applications\fg},
  \emph{ar{X}iv} (2011), no.~1103.3069v1.

\bibitem[Gre00]{G}
{\scshape C.~Greither} -- {\og Some cases of {B}rumer's conjecture for abelian
  {CM} extensions of totally real fields\fg}, \emph{Math. Z.} \textbf{233}
  (2000), no.~3, p.~515--534.

\bibitem[HK03]{HK}
{\scshape A.~Huber {\normalfont \smfandname} G.~Kings} -- {\og Bloch-{K}ato
  conjecture and {M}ain {C}onjecture of {I}wasawa theory for {D}irichlet
  characters\fg}, \emph{Duke Math. J.} \textbf{119} (2003), no.~3, p.~393--464.

\bibitem[IKY87]{IKY}
{\scshape Y.~Ihara, M.~Kaneko {\normalfont \smfandname} A.~Yukinari} -- {\og On
  some properties of the universal power series for {J}acobi sums\fg}, Galois
  representations and arithmetic algebraic geometry ({K}yoto, 1985/{T}okyo,
  1986), Adv. Stud. Pure Math., vol.~12, North-Holland, Amsterdam, 1987,
  p.~65--86.

\bibitem[Kak10]{Ka}
{\scshape M.~Kakde} -- {\og {T}he main conjecture of {I}wasawa theory for
  totally real fields\fg}, \emph{ar{X}iv} (2010), no.~1008.0142v3.

\bibitem[KNQDF96]{KNF}
{\scshape M.~Kolster, T.~Nguyen Quang~Do {\normalfont \smfandname}
  V.~Fleckinger} -- {\og Twisted {$S$}-units, {$p$}-adic class number formulas,
  and the {L}ichtenbaum conjectures\fg}, \emph{Duke Math. J.} \textbf{84}
  (1996), no.~3, p.~679--717.

\bibitem[KS95]{KS}
{\scshape J.~S. Kraft {\normalfont \smfandname} R.~Schoof} -- {\og Computing
  {I}wasawa modules of real quadratic number fields\fg}, \emph{Compositio
  Math.} \textbf{97} (1995), no.~1-2, p.~135--155, Special issue in honour of
  Frans Oort.

\bibitem[Mar]{M}
{\scshape A.~Martin} -- {\og {T}h{\'e}orie d'{I}wasawa des noyaux sauvages
  {\'e}tales et sommes de {G}au{\ss}, th{\`e}se de l'{U}niversit{\'e} de
  {F}ranche-{C}omt{\'e} 2011\fg}, \smfphdthesisname.

\bibitem[Nek06]{Ne}
{\scshape J.~Nekov{\'a}{\v{r}}} -- {\og Selmer complexes\fg},
  \emph{Ast\'erisque} (2006), no.~310, p.~viii+559.

\bibitem[Nic]{Nico}
{\scshape V.~Nicolas} -- {\og {C}onjecture {\'e}quivariante des nombres de
  {T}amagawa pour les motifs de {T}ate dans le cas ab{\'e}lien relatif,
  th{\`e}se de l'{U}niversit{\'e} de {F}ranche-{C}omt{\'e} 2013\fg},
  \smfphdthesisname.

\bibitem[Nic11]{Nick}
{\scshape A.~Nickel} -- {\og Leading terms of {A}rtin {$L$}-series at negative
  integers and annihilation of higher {$K$}-groups\fg}, \emph{Math. Proc.
  Cambridge Philos. Soc.} \textbf{151} (2011), no.~1, p.~1--22.

\bibitem[NQD05]{Ng}
{\scshape T.~Nguyen Quang~Do} -- {\og Conjecture {P}rincipale \'{E}quivariante,
  id\'eaux de {F}itting et annulateurs en th\'eorie d'{I}wasawa\fg}, \emph{J.
  Th\'eor. Nombres Bordeaux} \textbf{17} (2005), no.~2, p.~643--668.

\bibitem[NQDL06]{NLB}
{\scshape T.~Nguyen Quang~Do {\normalfont \smfandname} M.~Lescop} -- {\og
  Iwasawa descent and co-descent for units modulo circular units\fg},
  \emph{Pure Appl. Math. Q.} \textbf{2} (2006), no.~2, part 2, p.~465--496,
  With an appendix by J.-R. Belliard.

\bibitem[NQDN11]{NN}
{\scshape T.~Nguyen Quang~Do {\normalfont \smfandname} V.~Nicolas} -- {\og
  Nombres de {W}eil, sommes de {G}auss et annulateurs galoisiens\fg},
  \emph{Amer. J. Math.} \textbf{133} (2011), no.~6, p.~1533--1571.

\bibitem[PR94]{PR}
{\scshape B.~Perrin-Riou} -- {\og Th\'eorie d'{I}wasawa des repr\'esentations
  {$p$}-adiques sur un corps local\fg}, \emph{Invent. Math.} \textbf{115}
  (1994), no.~1, p.~81--161, With an appendix by Jean-Marc Fontaine.

\bibitem[RW02]{RW1}
{\scshape J.~Ritter {\normalfont \smfandname} A.~Weiss} -- {\og Toward
  equivariant {I}wasawa theory\fg}, \emph{Manuscripta Math.} \textbf{109}
  (2002), no.~2, p.~131--146.

\bibitem[RW11]{RW2}
\bysame , {\og On the ``main conjecture'' of equivariant {I}wasawa theory\fg},
  \emph{J. Amer. Math. Soc.} \textbf{24} (2011), no.~4, p.~1015--1050.

\bibitem[Sha11]{Sh}
{\scshape R.~Sharifi} -- {\og A reciprocity map and the two-variable {$p$}-adic
  {$L$}-function\fg}, \emph{Ann. of Math. (2)} \textbf{173} (2011), no.~1,
  p.~251--300.

\bibitem[Sna06]{Sn}
{\scshape V.~P. Snaith} -- {\og Stark's conjecture and new {S}tickelberger
  phenomena\fg}, \emph{Canad. J. Math.} \textbf{58} (2006), no.~2, p.~419--448.

\bibitem[Sol10]{So}
{\scshape D.~Solomon} -- {\og {S}ome new maps ad ideals in classical {I}wasawa
  theory with applications\fg}, \emph{ar{X}iv} (2010), no.~0905.4336v3.

\bibitem[Wit06]{W}
{\scshape M.~Witte} -- {\og On the equivariant main conjecture of {I}wasawa
  theory\fg}, \emph{Acta Arith.} \textbf{122} (2006), no.~3, p.~275--296.

\end{thebibliography}

\end{document}